\documentclass[12pt]{article}
\usepackage{amsfonts,amsmath,color,amsthm,amssymb}
\usepackage{mathtools}
\usepackage{tikz}
\usepackage{pstricks}
\usepackage{graphicx}
\usepackage{epstopdf}
\usepackage{permute}
\usepackage{latexsym}
\usepackage{mathrsfs}
\usepackage{hyperref}
\usepackage[utf8]{inputenc}
\usepackage{lscape}
\usepackage{longtable}
\usepackage{algorithm2e} 
\usepackage{enumitem} 
\usepackage{array}
\usepackage{tikz}

\usepackage{tikz-qtree,tikz-qtree-compat}
\usetikzlibrary{positioning,decorations.pathreplacing}

\title{The Expected Number of Distinct Substrings in an Alphabet String}

\author{Anant Godbole\\
East Tennessee State University and High Point University}

\newtheorem{thm}{Theorem}[section]
\newtheorem{lem}[thm]{Lemma}

\def\p{\mathbb P}
\def\e{\mathbb E}
\def\l{\lambda}
\def\lr{\left(}
\def\lg{{\rm lg\ }}
\def\rr{\right)}
\def\lc{\left\{}

\def\rc{\right\}}
\def\P{{\rm Po}}
\def\cl{{\mathcal L}}
\def\n{\noindent}
\def\nn{\nonumber}
\def\be{\begin{equation}}
\def\ee{\end{equation}}
\def\tv{d_{TV}}

\newcommand{\beq}{\begin{eqnarray}}
\newcommand{\eeq}{\end{eqnarray}}
\date{}
\begin{document}
\maketitle
\begin{abstract}
Consider a sequence of i.i.d.~trials $X=\{X_1, X_2, \ldots, X_n\}$ where $\p(X_i=j)={1}/{d}; j=1,2,\ldots, d$, or more generally $\p(X_i=j)=p_j; \sum_{1\le j\le d}p_j=1$.  We consider the variable $D$ that counts the number of distinct substrings of all lengths,$1\le k\le n$ in $X$ and prove results concerning $\e(D)$.  
\end{abstract}
\section{Introduction}   

\indent\indent Given a sequence of i.i.d.~trials $X=\{X_1, X_2, \ldots, X_n\}$ where $\p(X_i=j)={1}/{d}; j=1,2,\ldots, d$, or more generally $\p(X_i=j)=p_j; \sum_{1\le j\le d}p_j=1$, we consider the variable $D=D_n$ that counts the number of distinct substrings of all lengths $1\le k\le n$, in $X$ and prove results concerning $\e(D)$.    Of special interest, studied in Section 2, is the case where $d=2$ and $p_1=p_2=\frac{1}{2}$.  The general case (with equal $p_j$s) is studied in Section 3.

For example with the sequence
 
$$THTTHTT$$
 
\n we have the distinct (consecutive) strings being
$$T, H, TH, HT, TT, THT, HTT, TTH, THTT, HTTH, TTHT, THTT, THTTH, HTTHT,$$

$$ TTHTT, THTTHT, HTTHTT,\ {\rm and}\ THTTHTT,$$ 

\n for a total of $D=18$ (by convention we do not count the empty string $\emptyset$).  $D_n$ warrants study due to its interesting property that it is a measure of the intricacy of the sequence, and its first moment is a good place to begin such a study.  It is clear that the minimum value of $D$ is $n$, as exemplified by the string $11\ldots 1$, but it is not clear, at least to this author, what the maximum value of $D=\sum_{k=1}^nD_k$ is, though we see below that an upper bound on $D_k$ is $\min\lc(n-k+1),2^k\rc$. 

To put this paper in context, let us state that we are doing for words what was done in \cite{9pp} and \cite{swick} for permutations:  Improving on the work in \cite{9pp}, the following bound was proved in \cite{swick}:

\begin{thm} For sufficiently large $n$,
\[\e(D_n)\ge
\frac{n^2}{2}\lr1-\frac{17\ln n}{n}\rr.
\]
\end{thm}

Why is the lower bound in Theorem 1.1 relevant?  This is because in the context of permutations, 
\beq D=\sum_{k=1}^nD_k&\le&\sum_k\min\lc(n-k+1),k!\rc\nn\\&\le&\sum_k(n-k+1)\nn\\
&=&\sum_{k=1}^{n} k\nn\\&=& \frac{n^2}{2}(1+o(1))\eeq

As noted by Professor Koutras, moreover, the dependence structure in the case of words (due to repeated symbols) is much richer than for permutations.

For results when we look at embedded (non-consecutive) subsequences rather than consecutive strings, much more is known. For example, the string 10110 contains the subsequences 0, 1, 01, 10, 11, 00, 100, 101, 110, 111, 011, 010, 1011, 1010, 1110, 0110, and 10110.
What is the average case behavior?  In \cite{bgk}, it was proved that
\begin{thm}  Let $s_1,s_2,\ldots $ be a sequence of independent and identically distributed random variables with $\p(s_1=j)=\alpha_j, j=1,2,\ldots,d, \sum_j\alpha_j=1$.  Set $\alpha=(\alpha_1,\ldots,\alpha_d)$.  Let $\phi(S_n)$ be the number of distinct subsequences in $S_n=(s_1,\ldots,s_n)$.  Let $\psi(n)=E(\phi(S_n))$.  Then there exists $c=c_{d,\alpha}\ge 1$ such that
\[\psi(n)^{1/n}\to c\ {\rm as}\ n\to\infty,\] where $c=1$ iff $d\ge 1$ and $\max_j \alpha_j=1$.
\end{thm}

More exact results were proved for $d=2$, also in \cite{bgk}:
\begin{thm}
Suppose $\p[s_i=1]=\alpha \in [0,1]$ for all $1 \le i \le n$, and $\p[s_i=0]=1-\alpha$, $\alpha\ne 0,1$. Then we have
$$\psi(S_n)=\frac{A+B}{{2{\sqrt{\alpha(1-\alpha)}}}},$$
where  
$$A=A_n=\big(1-2\sqrt{\alpha(1-\alpha)}\big)\big(1-\big(1-\sqrt{\alpha(1-\alpha)}\big)^n\big)$$ and
$$B=B_n=\big(1+2\sqrt{\alpha(1-\alpha)}\big)\big(\big(1+\sqrt{\alpha(1-\alpha)}\big)^n-1\big).$$

\end{thm}

\n {\it Remark}: Theorem 1.3 was previously known for $\alpha=0.5$; It was shown in \cite{f}  by Flaxman et al.  (2004) that when Pr$[s_i=1]=.5$ then $\e[\phi(S_n)] \sim k({3}/{2})^n$ for a constant $k$. Later, Collins  \cite{c} improved this result by finding that $\e[\phi(S_n)]=2(3/2)^n-1$.  In \cite{bgk}, this result was generalized via Theorem 1.3 to non-uniform letter generation.   Two-state Markov chains were also considered.  (Interestingly, the maximum value of $\phi(S_n)$ was shown in \cite{f} to be $\gamma^n$, where $\gamma=1.618\ldots$ is the Golden Ratio).

Finally, we mention that a lot is known about tangentially related quantities such as suffix trees, tries (also called prefix trees) and distinct substrings  -- all from an algorithmic perspective

\section {Poisson Approximation and Consecutive Patterns:  The binary case}  

\begin{lem} The number of $D_k$ distinct consecutive strings of length $k$ satisfies
$D_k\le \min\{(n-k+1), 2^k\}$.
\end{lem}
\begin{proof} The number of distinct strings of length $k$ cannot be more than the total number of binary length $k$ strings, i.e., $2^k$.  However if $k$ is large, there might not be ``enough room" to accommodate all these stings.  Note that $n=2^k+k-1$ roughly when $k=\lg n$, where $\lg=\log_2$.  This proves Lemma 2.1.
\end{proof}
\begin{lem}
\[D\le \frac{n^2}{2}(1+o(1)).\]
\end{lem}
\begin{proof} Exactly the same as the proof of (1), with $k!$ being replaced by $2^k$.
\end{proof}
\n The main result of this paper is the following
\begin{thm}
The expected number $\e(D)=\e(D_n)$ of distinct consecutive words in $n$-long binary uniform letter generation satisfies:
\[\e(D)\ge \frac{n^2}{2}\lc1-\frac{6\lg n}{n}\rc= \frac{n^2}{2}\lc1-\frac{8.66\ln n}{n}\rc. \]
\end{thm}
\begin{proof}
Since a word contributes to the tally of distinct words iff it occurs at least once, it is clear that
\beq
\e(D)&=&\sum_k\e(D_k)\nn\\
&=&\sum_k\sum_{j=1}^{2^k} I({\rm the}\ j^{\rm th}\ {\rm word}\ N_j\ {\rm of}\ {\rm length}\ k\ {\rm appears}\  {\rm at}\ {\rm least}\ {\rm once}),\nn\\
&=&\sum_k\sum_{j=1}^{2^k} \p(U_{k,j}\ge 1)\nn\\
&\ge&\sum_k\sum_{j=1}^{A(k)} \p(U_{k,j}\ge 1),\eeq where 
\begin{itemize}\item $I(B)=1$ iff $B$ occurs ($I(B)=0$ otherwise); \item We have listed the words of length $k$ in some fashion, perhaps lexicographically, and labeled the $j$th word as $N_j$;  \item The number of occurrences of $N_j$ is denoted by $U_{k,j}$; \item $A(k)$ is the number of words of length $k$ with the initial and final segments of the word being equal and of lengths $\le k/2$ (e.g., with $k=8$, a candidate word would be $SFSSFSFS$). \end{itemize}  The strategy will be to prove that
\be\cl(U_{k,j})\approx\P(\e(U_{k,j})),\ee where for any variable $T$ we denote the distribution of $T$ by $\cl(T)$, and the Poisson distribution with parameter $\l$ by $\P(\l)$.    Note that $\e(U_{k,j})=(n-k+1)/2^k=\l$ for each $j$.   Equation (3) is proved via the total variation metric, i.e., 
by showing that for each $1\le j\le A(k)$, i.e., for each word for which the equality of its beginning and tail segments is of length $\le k/2$, \be
\tv(\cl(U_{k,j}), \P(\l)):=\sup_{A\subseteq {\mathbb Z}^+}\left\vert\p(U_{k,j}\in A)-\sum_{j\in A}\frac{e^{-\l}\l^j}{j!}\right\vert\le\varepsilon_{n,k}\to0,
\ee
where $\varepsilon_{n,k}$ does not depend on the word.  It would thus follow that for each $j$,
\be \p(U_{k,j}\ge 1)\ge(1-e^{-\l})-\varepsilon_{n,k},\ee
and thus via (2) that 
\be\e(D_k)\ge A(k)\cdot((1-e^{-\l})-\varepsilon_{n,k}).
\ee
We have that 
\be U_{k,j}=\sum_{j=1}^{n-k+1} I_j,\ee
where $I_j$ is the indicator variable that equals 1 if the word in question appears in the $k$ places $\{j,j+1,\ldots, j+k-1\}$ starting at $j$.  Also, $I_j$ is independent of the ensemble of $I_\ell$'s whose windows do not intersect those of $I_j$.  Thus Corollary 2.C.5 in \cite{bhj} indicates that
\begin{eqnarray}
&&\tv(\cl(U_{k,j}), \P(\l))\le\frac{1-e^{-\l}}{\l}\times\nonumber\\
&&\lr\sum_j\p^2(I_j=1)+\sum_j\sum_{i=j-k+1}^{j+k-1}[\p(I_jI_i=1)+\p(I_j=1)\p(I_i=1])\rr.\nonumber\\
&&
\end{eqnarray}
Since $\p(I_j=1)=\frac{1}{2^k}$ for each $j$, we have that $\l=\frac{n-k+1}{2^k}$.  Using this fact and bounding $\frac{1-e^{-\l}}{\l}$ by 1, (8) reduces to
\begin{eqnarray}
\tv(\cl(U_{k,j}), \P(\l))&\le&\frac{(n-k+1)}{2^{2k}}+\sum_j\sum_{i=j-k+1}^{j+k-1}\lr\p(I_jI_i=1)+\frac{1}{2^{2k}}\rr\nonumber\\
&\le&\frac{(n-k+1)}{2^{2k}}+\frac{2(n-k+1)k}{2^{2k}}+\sum_j\sum_{i=j-k+1}^{j+k-1}\p(I_jI_i=1).\nonumber\\
\end{eqnarray}
Equations (6) and (9) thus give
\begin{eqnarray}
\e(D_k)\nonumber&\ge&A(k)(1-e^{-\l})-\varepsilon_{n,k})\nonumber\\
&\ge& A(k)\times \bigg((1-e^{-\l})-\frac{(n-k+1)}{2^{2k}}-\nonumber\\
&&\frac{2(n-k+1)k}{2^{2k}}-\sum_j\sum_{i=j-k+1}^{j+k-1}\p(I_jI_i=1).\bigg)\end{eqnarray}
As estimate of $A(k)$ would be valuable at this point.  Assuming for simplicity that $k$ is even,  and letting $B_r$ denote those $k$-words whose initial and tail segments coincide and are of length $r$, 
\beq A(k) &=& 2^k - \left\vert\bigcup_{r=(k/2)+1}^{k-1} B_r\right\vert,\nn\\
&\ge& 2^k-\sum_r\vert B_r\vert\nn\\
&\ge& 2^k-C2^{k/2}\nn\\&=&2^k\lr1-\frac{C}{2^{k/2}}\rr,\eeq
so that $A(k)$ consists of almost all the words of length $k$.   Equation (11) is true because
of the following lemma.
\begin{lem}For $(k/2)+\le r\le k-1$, we have
\[\vert B_r\vert \le 2^{k-r}.\]
\end{lem}
\begin{proof} Starting with $r=k-1$, we see that $B_{k-1}=2$ since the segments must be $SSS\ldots SS$ or $FFF\ldots FF$.  Similarly, if $r=k-2$, then we must have the even elements of the string equaling each other, and so also there must be equality of the odd elements -- for a total of 4 degrees of freedom, with the strings being all $S$'s, all $F$'s, or alternating $S$ and $F$ (with two beginning possibilities).  This pattern continues, until, when $r=1$ the overlap element must be either $S$ or $F$ with the rest of the spots being free.  This completes the proof.
\end{proof}

We deal separately with the four terms in (10). 

\n First note that
\begin{eqnarray}
A(k)(1-e^{-\l})&\ge& A(k)\frac{\l}{1+\l}\nn\\
&\ge &A(k)\l(1-\l)\nn\\
&=&(n-k+1)\frac{A(k)}{2^k}-(n-k-1)^2\frac{A(k)}{2^{2k}}\nn\\
&\ge& (n-k+1)\lr1-\frac{C}{2^{k/2}}\rr-\frac{(n-k-1)^2}{2^k}\lr1-\frac{C}{2^{k/2}}\rr\nn\\
\end{eqnarray}
We retain the $(n-k+1)$ term in (12) for a later analysis and combine the two other terms above with the second and third terms in (10) to get:
\beq\e(D_k)
&\ge&(n-k+1)-(n-k+1)o(1)-\frac{(n-k+1)^2}{2^k}\lr1-\frac{C}{2^{k/2}}\rr\nn\\&-&\frac{3n^2}{2^{k}}-A(k)\sum_j\sum_{i=j-k+1}^{j+k-1}\p(I_iI_j=1)\nn\\
&\ge &(n-k+1)-(n-k+1)o(1)-\frac{n^2}{2^{k-1}}\nn\\
&&-A(k)\sum_j\sum_{i=j-k+1}^{j+k-1}\p(I_iI_j=1)
\eeq
The quantity $n^2/2^{{k-1}}$ above approaches 0 if $k=2\lg n+A_n$, where $A_n\to\infty$. 
Finally, letting $r$ denote the magnitude of the overlap between the ``$k$-windows" of $i$ and $j$, 
\begin{eqnarray}A(k)\times\sum_{j=1}^{n-k+1}\sum_{i=j-k+1}^{j+k-1}\p(I_jI_i=1)&\le&n2^{k+1}\times\sum_{r=1}^{k/2}\p(I_1I_{k+1-r}=1)\nonumber\\
&\le&n2^{k+1}\times\lc\frac{1}{2^{3k/2}}+\ldots+\frac{1}{2^{2k}}\rc\nonumber\\
&\le&C\frac{n}{2^{k/2}}\eeq
and thus 
\be\e(D_k)\ge (n-k+1)-(n-k-1)o(1)-\frac{n^2}{2^{k/2}}.\ee
Summing (15) from $k=3\lg n$ we get
\beq\e(D)&\ge& \sum_{k\ge 3\lg n}(n-k+1)-\sqrt{n}\nn\\
&=&\sum_{t=1}^{n-3\lg n}t-\sqrt{n}\nn\\
&=&\frac{(n-3\lg n)(n-3\lg n+1)}{2}-\sqrt{n}\nn\\
&\ge &\frac{n^2}{2}\lc1-\frac{6\lg n}{n}\rc.
\eeq
This proves Theorem 2.3.
\end{proof}
    \section{Poisson Approximation and Consecutive Patterns:  The i.i.d.~Uniform Case.} In this section,  for $d\ge 3$, we consider a sequence of i.i.d.~trials $X=\{X_1, X_2, \ldots, X_n\}$ where $\p(X_i=j)={1}/{d}; j=1,2,\ldots, d$.  
    
    First, Lemmas 2.1 and 2.2 go through unchanged, with $2^k$ being replaced by $d^k$.  
    
    The expected value of $D$ can be expressed as follows:
    
    \beq
\e(D)&=&\sum_k\e(D_k)\nn\\
&=&\sum_k\sum_{j=1}^{d^k} I({\rm the}\ j^{\rm th}\ {\rm word}\ N_j\ {\rm of}\ {\rm length}\ k\ {\rm appears}\  {\rm at}\ {\rm least}\ {\rm once}),\nn\\
&=&\sum_k\sum_{j=1}^{d^k} \p(U_{k,j}\ge 1)\nn\\
&\ge&\sum_k\sum_{j=1}^{A'(k)} \p(U_{k,j}\ge 1),\eeq 
where $A'(k)$ is the number of words of length $k$ with the initial and final segments of the word being equal and of lengths $r\le k/2$ (e.g., with $k=11, r=4, d=26$, a candidate word would be ABRACADABRA.)

As in the binary case, we will prove that

\be\e(D_k)\ge A'(k)\times((1-e^{-\l})-\varepsilon_{n,k}).
\ee
We have that 
\be U_{k,j}=\sum_{j=1}^{n-k+1} I_j,\ee
where $I_j$ is the indicator variable that equals 1 if the word in question appears in the $k$ places $\{j,j+1,\ldots, j+k-1\}$ starting at $j$.  Also, $I_j$ is independent of the ensemble of $I_\ell$'s whose windows do not intersect those of $I_j$.  Thus, as in Section 2,
\begin{eqnarray}
&&\tv(\cl(U_{k,j}), \P(\l))\le\frac{1-e^{-\l}}{\l}\times\nonumber\\
&&\lr\sum_j\p^2(I_j=1)+\sum_j\sum_{i=j-k+1}^{j+k-1}[\p(I_jI_i=1)+\p(I_j=1)\p(I_i=1)]\rr.\nonumber\\
&&
\end{eqnarray}
Since $\p(I_j=1)=\frac{1}{d^k}$ for each $j$, we have that $\l=\frac{n-k+1}{d^k}$.  Equation (20) reduces to
\begin{eqnarray}
\tv(\cl(U_{k,j}), \P(\l))&\le&\frac{(n-k+1)}{d^{2k}}+\sum_j\sum_{i=j-k+1}^{j+k-1}\lr\p(I_jI_i=1)+\frac{1}{d^{2k}}\rr\nonumber\\
&\le&\frac{(n-k+1)}{d^{2k}}+\frac{2(n-k+1)k}{d^{2k}}+\sum_j\sum_{i=j-k+1}^{j+k-1}\p(I_jI_i=1).\nonumber\\
\end{eqnarray}
Equations (18) and (21) thus give
\begin{eqnarray}
\e(D_k)&\ge& A'(k)((1-e^{-\l})-\varepsilon_{n,k})\nonumber\\
&\ge& A'(k)\times(1-e^{-\l})\nn\\
&-&A'(k)\lr\frac{(n-k+1)}{d^{2k}}+\frac{2(n-k+1)k}{d^{2k}}+\sum_j\sum_{i=j-k+1}^{j+k-1}\p(I_jI_i=1)\rr.\nn\\
\end{eqnarray}
As before, we estimate $A'(k)$ at this point.  Letting $B'_r$ denote those $k$-words whose initial and tail segments coincide and are of length $r$, 
\beq A'(k) &=& d^k - \left\vert\bigcup_{r=(k/2)+1}^{k-1} B'_r\right\vert,\nn\\
&\ge& d^k-\sum_r\vert B'_r\vert\nn\\
&\ge& d^k-Cd^{k/2}\nn\\&=&d^k\lr1-\frac{C}{d^{k/2}}\rr.\eeq
The ``feeder lemma" for (23) is
\begin{lem}For $(k/2)+1\le r\le k-1$, we have
\[\vert B'_r\vert \le d^{k-r}.\]
\end{lem}
\begin{proof} Exactly the same as that of Lemma 2.4, with $d$ replacing 2.  
\end{proof}

\n We deal separately with the four terms in (22):  

\begin{eqnarray}
A'(k)(1-e^{-\l})&\ge& A'(k)\frac{\l}{1+\l}\nn\\
&\ge &A'(k)\l(1-\l)\nn\\
&=&(n-k+1)\frac{A'(k)}{d^k}-(n-k-1)^2\frac{A'(k)}{d^{2k}}\nn\\
&\ge& (n-k+1)\lr1-\frac{C}{d^{k/2}}\rr-\frac{(n-k-1)^2}{d^k}\lr1-\frac{C}{d^{k/2}}\rr\nn\\
&\ge&(n-k+1)-(n-k+1)o(1)-\frac{(n-k-1)^2}{d^k}\lr1-\frac{C}{d^{k/2}}\rr.\nn\\ \eeq
The second and third terms in (22) combine to give
\be A'(k)\frac{(n-k+1)}{d^{2k}}+A'(k)\frac{2k(n-k+1)}{d^{2k}}\le\frac{3n^2A'(k)}{d^{2k}}\le\frac{3n^2}{d^k}.\ee
Thus
\beq
A'(k)\sum_j\sum_i\p(I_jI_i=1)\le nd^k.
\eeq
\beq\e(D_k)
&\ge&(n-k+1)-(n-k+1)o(1)-\frac{(n-k+1)^2}{d^k}\lr1-\frac{C}{d^{k/2}}\rr\nn\\&-&\frac{3n^2}{d^{k}}-A'(k)\sum_j\sum_{i=j-k+1}^{j+k-1}\p(I_iI_j=1)\nn\\
&\ge &(n-k+1)-(n-k+1)o(1)-\frac{n^2}{d^{k-1}}\nn\\
&&-A'(k)\sum_j\sum_{i=j-k+1}^{j+k-1}\p(I_iI_j=1).
\eeq
The quantity $n^2/d^{k-1}$ tends to zero if (say) $k\ge 2\lg n/\lg d+A_n$.  Finally Equation (22) gives
\beq A'(k)\sum_j\sum_{i=j-k+1}^{j+k-1}\p(I_iI_j=1)&\le&nd^{k+1}\sum_{r=1}^{k/2}\p(I_iI_{k+r-1}=1)\nn\\
&\le& nd^{k+1}\lc\frac{1}{d^{3k/2}}+\ldots \frac{1}{d^{2k}}\rc\nn\\
&\le& \frac{Cn}{d^{k/2}},\eeq
and thus 
\be\e(D_k)\ge (n-k+1)-\frac{n^2}{d^{k/2}}.\ee Summing (29) from $k=2\lg n/\lg d$ onwards, we get the following generalization of Theorem 2.3.
\beq\e(D)&\ge& \sum_{k\ge (2\lg n/\lg d}(n-k+1)-\sqrt{n}\nn\\
&=&\sum_{t=1}^{n-(2\lg n)/\lg d}t-\sqrt{n}\nn\\
&=&\frac{(n-B)(n-B+1)}{2}-\sqrt{n}\nn\\
&=&\frac{n^2}{2}\lc1-\frac{4\lg n}{n\lg d}\rc
\eeq
(where $B=2\lg n/\lg d$).
    \section{Open Questions}  
    
\indent\indent(a) Can we gain a deeper insight into the concentration of $D$ around $\e(D)$ by estimating the variance of $D$? 

\section{Funding}  This post-retirement research was not funded by any external funding agency. 

\section {Ethics Declarations}. There were no human subjects used. No data was used either.  Hence there was no need for IRB approval from the ETSU IRB Committee, https://www.etsu.edu/irb/  This is true for all the sections of this paper.

\end{document}